\documentclass{amsart}
\usepackage{amsfonts}
\usepackage{amsmath}
\usepackage{amssymb}
\usepackage{amsthm}

\newtheorem{theorem}{Theorem}

\newtheorem{lemma}{Lemma}
\newtheorem{rem}{Remark}

\begin{document}

\title[Orthogonal apartments in Hilbert Grassmannians]
{Orthogonal apartments in Hilbert Grassmannians. Finite-dimensional case}
\author{Mark Pankov}
\subjclass[2000]{}
\address{Department of Mathematics and Computer Science, 
University of Warmia and Mazury,
S{\l}oneczna 54, Olsztyn, Poland}
\email{pankov@matman.uwm.edu.pl}

\maketitle
\begin{abstract}
Let $H$ be a complex Hilbert space of finite dimension $n\ge 3$.
Denote by ${\mathcal G}_{k}(H)$ the Grassmannian consisting of 
$k$-dimensional subspaces of $H$.
Every orthogonal apartment of ${\mathcal G}_{k}(H)$ 
is defined by a certain orthogonal base of $H$ 
and consists of all $k$-dimensional subspaces spanned by subsets of this base.
For $n\ne 2k$ (except the case when $n=6$ and $k$ is equal to $2$ or $4$)
we show that every bijective transformation of ${\mathcal G}_{k}(H)$
sending orthogonal apartments to orthogonal apartments is induced by 
an unitary or conjugate-unitary operator on $H$.
The second result is the following: 
if $n=2k\ge 8$ and $f$ is a bijective transformation of ${\mathcal G}_{k}(H)$ such that 
$f$ and $f^{-1}$ send orthogonal apartments to orthogonal apartments
then there is an unitary or conjugate-unitary operator $U$ 
such that for every $X\in {\mathcal G}_{k}(H)$ we have $f(X)=U(X)$
or $f(X)$ coincides with the orthogonal complement of $U(X)$.
\end{abstract}

\section{Introduction and statement of results}
Let $H$ be a complex Hilbert space of finite or infinite dimension.
For every natu\-ral $k<\dim H$ we denote by ${\mathcal G}_{k}(H)$
the Grassmannian consisting of $k$-dimensional subspaces of $H$.
For every orthogonal base of $H$ 
the associated {\it orthogonal apartment} of ${\mathcal G}_{k}(H)$
consists of all $k$-dimensional subspaces spanned by subsets of this base.
Orthogonal apartments  of Hilbert Grassmannians were introduced in \cite{Pankov}.
This notion comes from the theory of Tits buildings \cite{Tits}.
Grassmannians related to a building of type $\textsf{A}_{n-1}$ 
are the Grassmannians ${\mathcal G}_{k}(V)$, $k\in \{1,\dots,n-1\}$
formed by $k$-dimensional subspaces of an $n$-dimensional vector space $V$
(see \cite{Pankov-book1,Pasini} for the details) 
and every apartment of ${\mathcal G}_{k}(V)$ consists of all $k$-dimensional subspaces
spanned by subsets of a certain base of $V$.

Recall that two closed subspaces $X,Y\subset H$ are {\it compatible}
if there exist closed subspaces $X',Y'$ such that $X\cap Y,X',Y'$
are mutually orthogonal and 
$$X=X'+(X\cap Y),\;\;Y=Y'+(X\cap Y).$$
The concept of orthogonal apartment is interesting for the following reason: 
orthogo\-nal apartments can be characterized as maximal subsets of 
mutually compatible elements of ${\mathcal G}_{k}(H)$ \cite[Proposition 1]{Pankov}. 
Note that the compatibility relation is defined for any logic, 
i.e. a lattice with an addition operation  known as the {\it negation}.
In classical logics any two elements are compatible and
quantum logics contain non-compatible elements.

Consider  the logic ${\mathcal L}(H)$ whose elements are closed subspaces of $H$ 
and the negation is the operation of orthogonal complementary.
In the case when $H$ is infinite-dimensional and separable,
this logic is exploited in mathematical foundations of quantum theory 
(see, for example, \cite{Var}). 
Every automorphism of ${\mathcal L}(H)$ is induced by 
an unitary or conjugate-unitary operator on $H$.
It follows from \cite[Theorem 2.8]{MS} that 
for every bijective transformation $f$ of ${\mathcal L}(H)$ preserving the compatibility relation
in both directions, i.e. $f$ and $f^{-1}$ send compatible elements to compatible elements, 
there is an unitary or conjugate-unitary operator $U$
such that for every $X\in {\mathcal L}(H)$ we have $f(X)=U(X)$ or $f(X)$
coincides with the orthogonal complement of $U(X)$.
Note that if closed subspaces $X$ and $Y$ are compatible
then the orthogonal complement of $X$ is compatible to $Y$.

If $H$ is infinite-dimensional and $f$ is a bijective transformation of ${\mathcal G}_{k}(H)$
such that $f$ and $f^{-1}$ send orthogonal apartments to orthogonal apartments
(equivalently, $f$ preserves the compatibility relation in both directions)
then $f$ is induced by an unitary or conjugate-unitary operator \cite[Theorem 1]{Pankov}.
In this paper, a similar result will be obtained for the case when $H$ is finite-dimensional. 
As in \cite{Pankov}, we will use {\it complementary subsets} of orthogonal apartments,
but the arguments  will be more complicated.  

\begin{theorem}\label{theorem1}
Suppose that $\dim H=n$ is finite and not less than $3$.
Let $f$ be a bijective transformation of ${\mathcal G}_{k}(H)$
sending orthogonal apartments to orthogonal apartments 
and let $n\ne 2k$.
We also require that $k$ is not equal to $2$ or $4$ if $n=6$.
Then $f$ is induced by an unitary or conjugate-unitary operator on $H$.
\end{theorem}

In Theorem \ref{theorem1} we do not require that 
the inverse transformation sends orthogonal apartments to orthogonal apartments. 
If $H$ is finite-dimensional and $f$ is a bijective transformation of ${\mathcal G}_{k}(H)$
sending compatible elements to compatible elements
then $f$ transfers orthogonal apartments to orthogonal apartments.
This is a simple consequence of the following facts: 
the class of orthogonal apartments coincides with 
the class of maximal subsets of mutually compatible elements of ${\mathcal G}_{k}(H)$, 
all orthogonal apartments in ${\mathcal G}_{k}(H)$ are of the same finite cardinality.

\begin{theorem}\label{theorem2}
Suppose that $\dim H=2k\ge 8$.
Let $f$ be a bijective transformation of ${\mathcal G}_{k}(H)$
such that $f$ and $f^{-1}$ send orthogonal apartments to orthogonal apartments,
in other words, $f$ preserves the compatibility relation in both directions.
Then there is an unitary or conjugate-unitary operator $U$ on $H$
such that for every $X\in {\mathcal G}_{k}(H)$ we have $f(X)=U(X)$
or $f(X)$ coincides with the orthogonal complement of $U(X)$.
\end{theorem}

\begin{rem}{\rm
Apartments preserving transformations of Grassmannians corresponding to buildings 
of classical types are described in \cite{Pankov-book1}. 
Also, there is a characteriza\-tion of isometric embeddings of Grassmann graphs 
in terms of ``generalized'' apartments \cite[Chapter 5]{Pankov-book2}.  
}\end{rem}

\section{Grassmann graphs}
To prove Theorems \ref{theorem1} and \ref{theorem2} 
we need some properties of Grassmann graphs.
Suppose that $\dim H=n$ is finite.
The {\it Grassmann graph} $\Gamma_{k}(H)$, $1<k<n-1$ 
is the graph whose vertex set is ${\mathcal G}_{k}(H)$ and 
two $k$-dimensional subspaces are adjacent vertices of this graph 
if their intersection is $(k-1)$-dimensional.
In what follows we say that two $k$-dimensional subspaces of $H$
are {\it adjacent} if they are adjacent vertices of $\Gamma_{k}(H)$.

If $S$ and $N$ are subspaces of $H$ such that $\dim S<k$ and $\dim N>k$
then we denote by $[S\rangle_{k}$ and $\langle N]_{k}$ the sets 
consisting of $k$-dimensional subspaces containing $S$ and contained in $N$,
respectively.
There are precisely the following two types of maximal cliques in $\Gamma_{k}(H)$:
\begin{enumerate}
\item[$\bullet$] the {\it stars} $[S\rangle_{k}$, $S\in {\mathcal G}_{k-1}(H)$;
\item[$\bullet$] the {\it tops} $\langle N]_{k}$, $N\in {\mathcal G}_{k+1}(H)$.
\end{enumerate}
See, for example, \cite[Section 3.2]{Pankov-book2}.

Let ${\mathcal A}$ be the orthogonal apartment of ${\mathcal G}_{k}(H)$
defined by an orthogonal base $B$. 
The restriction of $\Gamma_{k}(H)$ to ${\mathcal A}$
is isomorphic to the Johnson graph $J(n,k)$.
As above, we have two types of maximal cliques:
\begin{enumerate}
\item[$\bullet$] the {\it stars} ${\mathcal A}\cap [S\rangle_{k}$,
\item[$\bullet$] the {\it tops} ${\mathcal A}\cap\langle N]_{k}$,
\end{enumerate}
where $S\in {\mathcal G}_{k-1}(H)$ and $N\in {\mathcal G}_{k+1}(H)$ are spanned 
by subsets of $B$.
Stars and tops of ${\mathcal A}$ contain precisely $n-k+1$ and $k+1$ elements, respectively.

By classical Chow's theorem \cite{Chow}, every automorphism of the graph $\Gamma_{k}(H)$ 
is induced by an invertible semilinear operator if $n\ne 2k$.
In the case when $n=2k$, 
the automorphism group of $\Gamma_{k}(H)$ is generated by 
the automorphisms induced by invertible semilinear operators and 
the mapping $X\to X^{\perp}$, where $X^{\perp}$ is the orthogonal complement of $X$.
Huang's result \cite{Huang} (see also \cite[Section 3.10]{Pankov-book2}) says that
every bijective transformation of ${\mathcal G}_{k}(H)$ sending adjacent elements
to adjacent elements is an automorphism of $\Gamma_{k}(H)$.

\section{Complementary subsets in orthogonal apartments}

Let ${\mathcal A}$ be the orthogonal apartment of ${\mathcal G}_{k}(H)$
defined by an orthogonal base $\{e_{i}\}^{n}_{i=1}$ and let $1<k<n-1$.
For every $i\in \{1,\dots,n\}$ we denote by ${\mathcal A}(+i)$ and ${\mathcal A}(-i)$
the sets consisting of all elements of ${\mathcal A}$ which contain $e_{i}$ and 
do not contain $e_{i}$, respectively.
For any distinct $i,j\in \{1,\dots,n\}$ we define
$${\mathcal A}(+i,+j):={\mathcal A}(+i)\cap {\mathcal A}(+j),$$
$${\mathcal A}(+i,-j):={\mathcal A}(+i)\cap {\mathcal A}(-j),$$
$${\mathcal A}(-i,-j):={\mathcal A}(-i)\cap {\mathcal A}(-j).$$
A subset of ${\mathcal A}$ is called {\it inexact}
if there is an orthogonal apartment distinct from ${\mathcal A}$ and containing this subset.
By \cite[Lemma 1]{Pankov}, every maximal inexact subset is 
$${\mathcal A}(+i,+j)\cup {\mathcal A}(-i,-j).$$
We say that a subset ${\mathcal C}\subset {\mathcal A}$ is {\it complementary}
if ${\mathcal A}\setminus {\mathcal C}$ is a maximal inexact subset.
Then
$${\mathcal A}\setminus {\mathcal C}={\mathcal A}(+i,+j)\cup {\mathcal A}(-i,-j)$$
and
$${\mathcal C}={\mathcal A}(+i,-j)\cup {\mathcal A}(+j,-i).$$
This complementary subset is denoted by ${\mathcal C}_{ij}$.
Observe that ${\mathcal C}_{ij}={\mathcal C}_{ji}$.

\begin{rem}{\rm
Suppose that ${\mathcal A}$ is an arbitrary (not necessarily orthogonal) 
apartment of ${\mathcal G}_{k}(H)$. 
A subset of ${\mathcal A}$ is {\it inexact}
if there is an apartment distinct from ${\mathcal A}$ and containing this subset.
Every maximal inexact subset is of type ${\mathcal A}(+i,+j)\cup {\mathcal A}(-i)$
and the corresponding complimentary subset is ${\mathcal A}(+i,-j)$, 
see \cite[Section 5.2]{Pankov-book2}. 
}\end{rem}

\begin{lemma}\label{lemma1}
Let $f$ be a bijective transformation of ${\mathcal G}_{k}(H)$
sending orthogonal apartments to orthogonal apartments
and let ${\mathcal A}$ be an orthogonal apartment of ${\mathcal G}_{k}(H)$.
Then ${\mathcal C}\subset {\mathcal A}$ is a complementary subset if and only if
$f({\mathcal C})$ is a complementary subset of $f({\mathcal A})$.
\end{lemma}

\begin{proof}
 It is clear that $f$ transfers  inexact subsets of ${\mathcal A}$
to inexact subsets of $f({\mathcal A})$.
An inexact subset is maximal if and only if
it contains $$\binom{n-2}{k-2}+\binom{n-2}{k}$$ 
elements.
This implies that maximal inexact subsets of ${\mathcal A}$ 
go to maximal  inexact subsets of $f({\mathcal A})$.
Since ${\mathcal A}$ and $f({\mathcal A})$ have the same number of such subsets, 
${\mathcal X}$ is a maximal inexact subset of ${\mathcal A}$ if and only if 
$f({\mathcal X})$ is a maximal inexact subset of $f({\mathcal A})$.
This gives the claim.
\end{proof}

\begin{lemma}\label{lemma2}
Let $X,Y$ be distinct elements of an orthogonal apartment 
${\mathcal A}\subset {\mathcal G}_{k}(H)$. 
If $\dim(X\cap Y)=m$ then there are precisely
$$c(m)=(k-m)^{2}+m(n-2k+m)$$
distinct complementary subsets of ${\mathcal A}$ containing this pair.
\end{lemma}

\begin{proof}
Since $\dim(X\cap Y)=m$, we have $\dim(X+Y)=2k-m$.
Let $\{e_{i}\}^{n}_{i=1}$ be one of the orthogonal bases associated to ${\mathcal A}$.
If the complementary subset ${\mathcal C}_{ij}$ contains both $X$ and $Y$ then 
one of the following possibilities is realized:
\begin{enumerate}
\item[(1)] one of $e_{i},e_{j}$ belongs to $X\setminus Y$ and the other to 
$Y\setminus X$,
\item[(2)] one of $e_{i},e_{j}$ belongs to $X\cap Y$ and the other is not contained in $X+Y$.
\end{enumerate}
There are precisely $(k-m)^2$ and $m(n-2k+m)$ distinct ${\mathcal C}_{ij}={\mathcal C}_{ji}$
satisfying $(1)$ and $(2)$, respectively.
\end{proof}

\begin{rem}{\rm
In the case when $H$ is infinite-dimensional, 
there is the following characterization of the orthogonality relation \cite[Lemma 2]{Pankov}:
two elements in an orthogonal apartment ${\mathcal A}$ are orthogonal if and only if 
there is only a finite number of complementary subsets of ${\mathcal A}$
containing this pair.
This implies that every bijective transformation of ${\mathcal G}_{k}(H)$
preserving the class of orthogonal apartments in both directions 
preserves also the orthogonality relation in both directions (see \cite{Pankov} for the details).
By \cite{Cyory,Semrl}, the latter condition guarantees that 
the bijection is induced by an unitary or conjugate-unitary operator.
}\end{rem}

\section{Proof of Theorem \ref{theorem1}}
Suppose that $\dim H=n$ is finite and not less then $3$.
Let $f$ be a bijective transformation of ${\mathcal G}_{k}(H)$
sending orthogonal apartments to orthogonal apartments.

\subsection{Preliminary remarks}
The mapping  $X\to X^{\perp}$ is a bijection between 
${\mathcal G}_{k}(H)$ and ${\mathcal G}_{n-k}(H)$.
It transfers every orthogonal apartment of ${\mathcal G}_{k}(H)$
to the orthogonal apartment of ${\mathcal G}_{n-k}(H)$ defined by the same orthogonal base.
Thus $X\to f(X^{\perp})^{\perp}$ is a bijective transformation of ${\mathcal G}_{n-k}(H)$
sending orthogonal apartments to orthogonal apartments.
If it is induced by an unitary or conjugate-unitary operator then 
$f$ is induced by the same operator.
For this reason it is sufficiently to prove Theorem \ref{theorem1}
only in the case when $k\le n-k$.

Suppose that $k=1$.
Then $f$ transfers orthogonal elements of ${\mathcal G}_{1}(H)$ to orthogonal elements.
For any $2$-dimensional subspace $Y\subset H$ 
we take orthogonal $1$-dimensional subspaces $X_{1},X_{2}\subset Y$
and extend them to an orthogonal apartment $\{X_{i}\}^{n}_{i=1}$ in ${\mathcal G}_{1}(H)$.
If  $X$ is a $1$-dimensional subspace of $Y$ then $f(X)$ is orthogonal
to $f(X_{i})$ for every $i\ge 3$, i.e. $f(X)$ is contained in $f(X_1)+f(X_2)$.
So, $f$ sends all lines of the projective space over $H$ to 
subsets of lines and, by the Fundamental Theorem of Projective Geometry \cite{Artin},
$f$ is induced by an invertible semilinear operator. 
This operator transfers orthogonal vectors to orthogonal vectors.
An easy verification shows that it is a scalar multiple of 
an unitary or conjugate-unitary operator $U$.
It is clear that $f$ is induced by $U$.

From this moment we will suppose that $1<k\le n-k$.

\subsection{The case $n\ne 2k+2$}
Following Lemma \ref{lemma2}, we consider the quadratic function 
$$c(x)=(k-x)^{2}+x(n-2k+x)=2x^{2}-(4k-n)x+k^2.$$
It takes the minimal value on $x=\frac{4k-n}{4}$.
This implies that
$$c(k-1)>c(m)$$ 
for all $m\in\{0,1,\dots,k-2\}$ if
$$k-1>\frac{4k-n}{2}$$
or, equivalently, if 
$$k<\frac{n-2}{2}.$$
By our assumption, $n=2k+l$ for some natural $l\ge 0$
and the latter inequality fails only in the case when $l\in\{0,1,2\}$.

If $n=2k+2$ then $k-1$ is equal to $(4k-n)/2$ which means that 
$$c(k-1)=c(0)$$
and the latter number is greater than $c(m)$
for any $m\in \{1,\dots,k-2\}$.

If $n=2k+1$ then $k-1$ is less than $(4k-n)/2$ which implies that $c(k-1)<c(0)$,
but we have $$c(k-1)\ne c(m)$$ for every $m\in\{0,1,\dots,k-2\}$.
Indeed, if $c(k-1)=c(x)$ and $x\ne k-1$ then an easy calculation shows that 
$x=1/2$.

Lemma \ref{lemma2} together with the above arguments give the following
characterization of adjacent elements in orthogonal apartments.

\begin{lemma}\label{lemma1-1}
Suppose that $n$ is not equal to $2k$ or $2k+2$.
Let ${\mathcal A}$ be an orthogonal apartment of ${\mathcal G}_{k}(H)$.
Then $X,Y\in{\mathcal A}$ are adjacent  if and only if 
there are precisely  $c(k-1)$
distinct complementary subsets of ${\mathcal A}$ containing this pair.
\end{lemma}

\begin{lemma}\label{lemma1-2}
As in the previous lemma, we suppose that $n$ is not equal to $2k$ or $2k+2$
and ${\mathcal A}$ is an orthogonal apartment of ${\mathcal G}_{k}(H)$.
Then $X,Y\in {\mathcal A}$ are adjacent 
if and only if $f(X)$ and $f(Y)$ are adjacent; 
moreover, $f$ transfers stars of ${\mathcal A}$ to stars of $f({\mathcal A})$.
\end{lemma}

\begin{proof}
Using Lemmas \ref{lemma1} and \ref{lemma1-1} we show that 
$X,Y\in {\mathcal A}$ are adjacent if and only if the same holds for $f(X)$ and $f(Y)$.
Then $f$ transfers every star of ${\mathcal A}$ to a star or a top of $f({\mathcal A})$.
Stars and tops contain $n-k+1$ and $k+1$ elements, respectively.  
Since $n\ne 2k$, these numbers are distinct and 
the image of every star is a star.
\end{proof}

We prove Theorem \ref{theorem1} for $n\ne 2k+2$.

Let $X$ and $Y$ be adjacent elements of ${\mathcal G}_{k}(H)$ 
which are not contained in an orthogonal apartment.
Denote by $N$ the orthogonal complement of $X\cap Y$. 
The dimension of $N$ is equal to $n-k+1$. 
Let $S$ be the intersection of $X+Y$ with $N$.
This is a $2$-dimensional subspace.  
Then
$$\dim (S^{\perp}\cap N)=n-k-1\ge 2$$
(since $n=2k+l$ and $l\ge 1$, we have $n-k-1=k+l-1\ge k\ge 2$).
This implies the existence of orthogonal $1$-dimensional subspaces 
$P,Q\subset N$ which are orthogonal to $S$.
We set 
$$X'=P+(X\cap Y)\;\mbox{ and }\;Y'=Q+(X\cap Y).$$
Then $X,X',Y'$ are mutually compatible and the same holds for $Y,X',Y'$.
Let ${\mathcal A}$ and ${\mathcal A}'$ be orthogonal apartments 
containing $X,X',Y'$ and $Y,X',Y'$, respectively.
Since $X,X',Y'$ are contained in a star of ${\mathcal A}$,
Lemma \ref{lemma1-2} guarantees that $f(X),f(X'),f(Y')$ 
are contained in a star of $f({\mathcal A})$. 
Similarly, we establish that $f(Y),f(X'),f(Y')$ are in a star of $f({\mathcal A}')$. 
Therefore, $f(X)$ and $f(Y)$ both contain the $(k-1)$-dimensional subspace $f(X')\cap f(Y')$
which implies that they are adjacent.

So, $f$ sends adjacent vertices of $\Gamma_{k}(H)$ to adjacent vertices
which implies that $f$ is an automorphism of $\Gamma_{k}(H)$.
Then $f$ is induced by an invertible semilinear operator. 
This operator sends orthogonal vectors to orthogonal vectors.
Hence it is a scalar multiple of an unitary or conjugate-unitary operator $U$.
The transformation $f$ is induced by $U$.

\subsection{The case $n=2k+2$}
Suppose that $n=2k+2$ and consider an orthogonal apartment 
${\mathcal A}\subset {\mathcal G}_{k}(H)$.
By the previous subsection, if $X,Y\in {\mathcal A}$ and there are 
precisely $c(k-1)$ distinct complementary subsets of ${\mathcal A}$ containing this pair
then $X$ and $Y$ are adjacent or $\dim(X\cap Y)=0$ and they are orthogonal.

We say that two distinct complementary subsets ${\mathcal C}_{ij}$ and ${\mathcal C}_{i'j'}$
are {\it adjacent} if $\{i,j\}\cap \{i',j'\}\ne \emptyset$. 
In this case, we have
$$|{\mathcal C}_{ij}\cap {\mathcal C}_{i'j'}|=\binom{n-3}{k-1}+\binom{n-3}{k-2}=
\binom{n-2}{k-1}=\binom{2k}{k-1};$$
otherwise, we get
$$|{\mathcal C}_{ij}\cap {\mathcal C}_{i'j'}|=4\binom{n-4}{k-2}=4\binom{2k-2}{k-2}.$$
The equality 
$$\binom{2k}{k-1}=4\binom{2k-2}{k-2}$$
holds only for $k=2$, i.e. $n=6$.
Therefore, for $n\ne 6$ two complementary subsets 
${\mathcal C},{\mathcal C}'\subset {\mathcal A}$ are adjacent 
if and only if $f({\mathcal C})$ and $f({\mathcal C}')$ are adjacent complementary subsets 
of $f({\mathcal A})$.

If $X,Y\in {\mathcal A}$ are orthogonal then 
for every complementary subset ${\mathcal C}\subset {\mathcal A}$ containing this pair 
there is a complementary subset of ${\mathcal A}$ containing $X,Y$ and 
adjacent to ${\mathcal C}$.
In the case when $X,Y\in {\mathcal A}$ are adjacent,
there is the unique complementary subset of ${\mathcal A}$ containing $X,Y$
and non-adjacent to any other complementary subset of ${\mathcal A}$ containing $X,Y$.

Using the latter observation, 
we establish the direct analogue of Lemma \ref{lemma1-2} for $n=2k+2\ne 6$.
As in the previous subsection, we show that 
$f$ is induced by an unitary or conjugate-unitary operator. 

\begin{rem}\label{rem3}{\rm
Consider the case when $n=6$ and $k=2$.
If ${\mathcal A}$ is an orthogonal apartment of ${\mathcal G}_{k}(H)$
then any distinct $X,Y\in {\mathcal A}$ are contained 
in precisely $4$ distinct complementary subsets of ${\mathcal A}$.
The intersection of any two distinct complementary subsets consists of $3$ elements. 
Thus the dimension of the intersection of $X,Y\in {\mathcal A}$
cannot be determined in terms of complementary subsets. 
}\end{rem}

\section{Proof of Theorem \ref{theorem2}}

Suppose that $\dim H=2k\ge 8$ and $f$ is a bijective transformation of ${\mathcal G}_{k}(H)$
such that $f$ and $f^{-1}$ send orthogonal apartments to orthogonal apartments.

In this case we have
$$c(m)=(k-m)^{2}+m^2.$$
It is easy to see that 
$$c(0)>c(m)$$
for every $m\in \{1,\dots,k-1\}$
and 
$$c(m)=c(m')$$
only in the case when $m'=m$ or $m'=k-m$.
The standard arguments give the following.

\begin{lemma}\label{lemma2-1}
Let $X,Y$ be distinct elements in an  orthogonal apartment 
${\mathcal A}\subset {\mathcal G}_{k}(H)$.
If $\dim (X\cap Y)=m$ then 
$\dim (f(X)\cap f(Y))$ is equal to $m$ or $k-m$.
\end{lemma}

Since $\dim H=2k$, the orthogonal complement $X^{\perp}$
is the unique element of ${\mathcal G}_{k}(H)$ orthogonal to $X\in {\mathcal G}_{k}(H)$.
Any pair $X,X^{\perp}$ is contained in a certain orthogonal apartment of ${\mathcal G}_{k}(H)$.
In the previous lemma we put $m=0$ and get the following.

\begin{lemma}\label{lemma2-2}
For every $X\in {\mathcal G}_{k}(H)$ we have $f(X^{\perp})=f(X)^{\perp}$.
\end{lemma}

Let ${\mathcal G}'$ be the set of all $2$-element subsets 
$\{X,X^{\perp}\}\subset{\mathcal G}_{k}(H)$.
By Lemma \ref{lemma2-2}, $f$ induces a bijective transformation of ${\mathcal G}'$.
We denote this transformation by $f'$.

Consider the graph $\Gamma'$ whose vertex set is ${\mathcal G}'$
and subsets 
$$\{X,X^{\perp}\},\{Y,Y^{\perp}\}\in {\mathcal G}'$$ 
are adjacent vertices of this graph if $X$ is adjacent to $Y$ or $Y^{\perp}$.
The latter conditions guarantees that $X^{\perp}$ is adjacent to 
$Y^{\perp}$ or $Y$, respectively.
Two elements of ${\mathcal G}'$ will be called {\it adjacent}
if they are adjacent vertices of $\Gamma'$.

For every $(k-1)$-dimensional subspace $S\subset H$
we denote by ${\mathcal C}(S)$ the set of all $\{X,X^{\perp}\}\in {\mathcal G}'$
such that $X$ or $X^{\perp}$ contains $S$ 
(then the $(k+1)$-dimensional subspace $S^{\perp}$ contains $X^{\perp}$ or $X$, respectively). 
This is a maximal clique of the graph $\Gamma'$.

\begin{lemma}\label{lemma2-3}
If ${\mathcal C}$ is a maximal clique of $\Gamma'$
then ${\mathcal C}={\mathcal C}(S)$ for a certain $(k-1)$-dimen\-sional subspace $S$.
\end{lemma}

\begin{proof}
Let $\{X,X^{\perp}\}\in{\mathcal C}$.
Denote by ${\mathcal X}$ the set of all $Y\in {\mathcal G}_{k}(H)$ adjacent to $X$ 
and such that $\{Y,Y^{\perp}\}\in{\mathcal C}$.
It is clear that ${\mathcal X}$ is a maximal clique of $\Gamma_{k}(H)$.
If this is the star corresponding to a $(k-1)$-dimensional subspace $S$
then ${\mathcal C}={\mathcal C}(S)$.
If ${\mathcal X}$ is the top defined by a $(k+1)$-dimensional subspace $N$
then $N^{\perp}$ is $(k-1)$-dimensional and 
${\mathcal C}={\mathcal C}(N^{\perp})$.
\end{proof}

\begin{lemma}\label{lemma2-4}
If $X,Y,Z$ are mutually adjacent elements of ${\mathcal G}_{k}(H)$ contained in 
an orthogonal apartment then 
there is the unique maximal clique of $\Gamma'$ containing 
the corresponding elements of ${\mathcal G}'$.
\end{lemma}

\begin{proof}
For $S=X\cap Y \cap Z$ and $N=X+Y+Z$ one of the following possibilities is realized:
\begin{enumerate}
\item $\dim S=k-1$ and $\dim N=k+2$,
\item $\dim S=k-2$ and $\dim N=k+1$.
\end{enumerate}
The required maximal clique is ${\mathcal C}(S)$ or ${\mathcal C}(N^{\perp})$,
respectively.
\end{proof}

\begin{lemma}\label{lemma2-5}
The transformation $f'$ is an automorphism of the graph $\Gamma'$. 
\end{lemma}

\begin{proof}
Let $\{X,X^{\perp}\}$ and $\{Y,Y^{\perp}\}$ be adjacent elements of ${\mathcal G}'$.
We need to show that $f'$ transfers them to adjacent elements.
It is sufficient to restrict ourself to the case when 
$X$ and $Y$ are adjacent.

Suppose that there is an orthogonal apartment of ${\mathcal G}_{k}(H)$ containing $X$ and $Y$. 
Note that $X^{\perp}$ and $Y^{\perp}$ also belong to this apartment.   
By Lemma \ref{lemma2-1}, $f(X)$ and $f(Y)$ are adjacent or
$$\dim(f(X)\cap f(Y))=1.$$
The latter equality implies that $f(X)$ is adjacent to $f(Y)^{\perp}=f(Y^{\perp})$ 
Therefore, $f'$ transfers $\{X,X^{\perp}\}$ and $\{Y,Y^{\perp}\}$
to adjacent elements of ${\mathcal G}'$.

Consider the case when there is no orthogonal apartment containing $X$ and $Y$.
Let $N$ be the orthogonal complement of $X\cap Y$
and let $S$ be the intersection of $X+Y$ with $N$.
The dimensions of $N$ and $S$ are equal to $k+1$ and $2$,
respectively.
Also, we have
$$\dim (S^{\perp}\cap N)=k-1\ge 3.$$
This means that $N$ contains three mutually orthogonal $1$-dimensional subspaces 
$P,Q,T$ which are orthogonal to $S$.
We set 
$$X'=P+(X\cap Y),\;Y'=Q+(X\cap Y),\;Z'=T+(X\cap Y).$$
Note that $X,Y,X',Y',Z'$ are mutually adjacent.
Since $X, X',Y',Z'$ are mutually compatible,
there is an orthogonal apartment of ${\mathcal G}_{k}(H)$ containing them. 
The same holds for $Y, X',Y',Z'$.
Then there is an orthogonal apartment containing $$f(X), f(X'),f(Y'),f(Z')$$ and,
by the above arguments,
the corresponding elements of ${\mathcal G}'$ are mutually adja\-cent.
Lemma \ref{lemma2-4} implies the existence of the unique maximal clique 
$${\mathcal C}(M),\;M\in {\mathcal G}_{k-1}(H)$$ containing them.
Similarly, there is the unique maximal clique 
$${\mathcal C}(N),\;N\in {\mathcal G}_{k-1}(H)$$
which contains the elements of ${\mathcal G}'$ corresponding to  
$$f(Y), f(X'),f(Y'),f(Z').$$
Then ${\mathcal C}(M)\cap {\mathcal C}(N)$ contains at least $3$ elements
and Lemma \ref{lemma2-4} guarantees that $M=N$.
So, the elements of ${\mathcal G}'$ corresponding to $f(X)$ and $f(Y)$
belong to a certain maximal clique of $\Gamma'$, in other words,
they are adjacent.

Similarly, we show that $f^{-1}$ sends adjacent elements to adjacent elements.
\end{proof}

It follows from Lemma \ref{lemma2-5} that $f'$ sends maximal cliques of $\Gamma'$ to
maximal cliques and there is a transformation $g$ of ${\mathcal G}_{k-1}(H)$
such that 
$$f'({\mathcal C}(S))={\mathcal C}(g(S))$$
for every $S\in {\mathcal G}_{k-1}(H)$.
This transformation is bijective. 

\begin{lemma}\label{lemma2-6}
The transformation $g$ sends orthogonal apartments to orthogonal apartments.
\end{lemma}

\begin{proof}
Let $B$ be an orthogonal base of $H$
and let ${\mathcal A}$ be the associated orthogonal apartment of ${\mathcal G}_{k}(H)$.
We take any orthogonal base $B'$ corresponding to 
the orthogonal apartment $f({\mathcal A})$.
An easy verification shows that $g$ transfers 
the orthogonal apartment defined by $B$ to the orthogonal apartment defined by $B'$.
\end{proof}

By Theorem \ref{theorem1}, the transformation $g$ 
is induced by an unitary or conjugate-unitary operator $U$. 
This operator satisfies the required condition.


\begin{thebibliography}{999}

\bibitem{Artin} Artin E., {\it Geometric algebra}, Interscience Publishers 1957.

\bibitem{Chow}
Chow W. L.,  {\it On the geometry of algebraic homogeneous spaces}, 
Ann. of Math. 50(1949), 32--67.

\bibitem{Cyory}
Gy\"ory M.,
{\it Transformations on the set of all $n$-dimensional subspaces of a
Hilbert space preserving orthogonality}, Publ. Math. Debrecen 65 (2004),
233--242.


\bibitem{Huang}
Huang W.-l.,
{\it Adjacency preserving transformations of Grassmann spaces},
Abh. Math. Sem. Univ. Hamburg 68(1998), 65--77.

\bibitem{MS} Moln\'ar L., \v{S}emrl P.,
{\it Transformations of the unitary group on a Hilbert space},
J. Math. Annal. Appl. 388(2012), 1205--1217.

\bibitem{Pankov-book1} Pankov M.,
{\it Grassmannians of classical buildings}, World Scientific 2010.

\bibitem{Pankov-book2} Pankov M.,
{\it Geometry of semilinear embeddings: Relations to graphs and codes},
World Scientific 2015.

\bibitem{Pankov}
Pankov M., {\it Orthogonal apartments in Hilbert Grassmannians},
arXiv 2015.

\bibitem{Pasini} 
Pasini A., {\it Diagram Geometries}, Oxford Science Publications,
Clarendon/Oxford University Press 1994. 

\bibitem{Semrl}
\v{S}emrl P., {\it Orthogonality preserving transformations on the set of 
$n$-dimensional subspaces of a Hilbert space}, Illinois J. Math. 48 (2004), 
567--573.

\bibitem{Tits}
Tits J.,
{\it Buildings of spherical type and finite BN-pairs},
Lecture Notes in Mathematics 386,  Springer 1974.

\bibitem{Var} Varadarajan V.S. {\it Geometry of Quantum Theory} (2nd edition),
Springer 2000.
\end{thebibliography}
\end{document}